\documentclass[12pt]{article}
\usepackage{amsmath,amssymb,amscd,amsthm}
\usepackage{color}
\usepackage[dvipdfmx,colorlinks,setpagesize=false]{hyperref}
\usepackage{mathtools}
\mathtoolsset{showonlyrefs}
\usepackage{color}

\newcounter{myitem}[section]

\newtheorem{thm}{Theorem}

\newtheorem{cor}[thm]{Corollary}

\theoremstyle{definition}

\newtheorem{rmk}[thm]{Remark}

\numberwithin{equation}{thm}

\newcommand{\RomI}{\uppercase\expandafter{\romannumeral 1}}
\newcommand{\RomII}{\uppercase\expandafter{\romannumeral 2}}
\newcommand{\RomIII}{\uppercase\expandafter{\romannumeral 3}}

\newcommand{\cA}{\mathcal A}

\newcommand{\cH}{\mathcal H}

\newcommand{\cO}{\mathcal O}

\newcommand{\ca}{complex analytic }

\newcommand{\coh}{{\rm H}}

\DeclareMathOperator{\dbar}{\overline{\partial}}

\DeclareMathOperator{\kernel}{Ker}

\begin{document}

\title{A remark on the decomposition theorem
for direct images of canonical sheaves
tensorized with semipositive vector bundles}
\author{Taro Fujisawa \\
 \\
Tokyo Denki University \\
e-mail: fujisawa@mail.dendai.ac.jp}
\footnotetext[0]
{\hspace{-18pt}2010 {\itshape Mathematics Subject Classification}.
Primary 14F05; Secondary 32J25. }

\maketitle

\begin{abstract}
The purpose of this short note is to give a remark
on the decomposition theorem
for direct images of canonical sheaves
tensorized with Nakano semipositive vector bundles.
Although our result is a direct consequence of Takegoshi's work
in \cite{Takegoshi},
it was not stated explicitly in \cite{Takegoshi}.
Here we give the precise statement and the proof.
\end{abstract}

The decomposition theorem for direct images of canonical sheaves
was proved by J.~Koll\'ar \cite[Theorem 3.1]{KollarII}.
Inspired by the work of S. Matsumura \cite{MatsumuraVT},
here we note that
the decomposition theorem also holds
for direct images of canonical sheaves
tensorized with Nakano semipositive vector bundles.
Although Theorem \ref{main theorem} below
is a direct consequence of Takegoshi's results in \cite{Takegoshi},
it was not stated explicitly there.
Therefore we give the precise statement
of the decomposition theorem and prove it explicitly here.
We remark that Theorem \ref{main theorem} below
immediately implies 
the weaker form of the decomposition theorem
\cite[\RomI \ Decomposition Theorem]{Takegoshi}
(cf. Corollary \ref{corollary}).

The author would like to thank Osamu Fujino and Shin-ichi Matsumura
for their helpful discussion.

\begin{thm}
\label{main theorem}
Let $X$ be a K\"ahler manifold of pure dimension,
$Y$ a \ca space
and $f: X \longrightarrow Y$ a proper surjective morphism
such that all the connected components of $X$ are mapped surjectively to $Y$.
For a Nakano semipositive vector bundle $(E,h)$ on $X$,
we have an isomorphism
\begin{equation}
\bigoplus_{q}R^qf_{\ast}(\omega_X \otimes E)[-q]
\simeq
Rf_{\ast}(\omega_X \otimes E)
\end{equation}
in the derived category of $\cO_Y$-modules.
\end{thm}
\begin{proof}
The sheaf of $E$-valued $C^{\infty}$ $(p,q)$-forms on $X$
is denoted by $\cA_X^{p,q}(E)$.
Then we have the Dolbeault quasi-isomorphism
\begin{equation}
\omega_X \otimes E
\longrightarrow
(\cA_X^{n,\bullet}(E), \dbar) ,
\end{equation}
which is an $f_{\ast}$-acyclic resolution of $\omega_X \otimes E$.
Therefore we have an isomorphism
\begin{equation}
Rf_{\ast}(\omega_X \otimes E)
\simeq
(f_{\ast}\cA_X^{n,\bullet}(E), \dbar)
\end{equation}
in the derived category of $\cO_Y$-modules.

In the proof of Theorem 6.4 in \cite{Takegoshi},
Takegoshi defines an $\cO_Y$-subsheaf
$R^0f_{\ast}\cH^{n,q}(E)$ of
$\kernel(\dbar:
f_{\ast}\cA_X^{n,q}(E)
\longrightarrow
f_{\ast}\cA_X^{n,q+1}(E))$
such that the canonical inclusion
\begin{equation}
R^0f_{\ast}\cH^{n,q}(E)
\longrightarrow
\kernel(\dbar:
f_{\ast}\cA_X^{n,q}(E)
\longrightarrow
f_{\ast}\cA_X^{n,q+1}(E))
\end{equation}
induces an isomorphism of $\cO_Y$-modules
\begin{equation}
\label{isomorphism for q}
R^0f_{\ast}\cH^{n,q}(E)
\overset{\simeq}{\longrightarrow}
R^qf_{\ast}(\omega_X \otimes E)
\end{equation}
for every $q$.
The composite of the inclusions
\begin{equation}
R^0f_{\ast}\cH^{n,q}(E)
\longrightarrow
\kernel(\dbar:
f_{\ast}\cA_X^{n,q}(E)
\longrightarrow
f_{\ast}\cA_X^{n,q+1}(E))
\longrightarrow
f_{\ast}\cA_X^{n,q}(E)
\end{equation}
is denoted by $\varphi^q$.
Then $\varphi^q$ defines a morphisms of complexes
\begin{equation}
R^0f_{\ast}\cH^{n,q}(E)[-q]
\longrightarrow
f_{\ast}\cA_X^{n,\bullet}(E)
\end{equation}
for every $q$.
Because we have the isomorphism \eqref{isomorphism for q} for every $q$,
we obtain a quasi-isomorphism
\begin{equation}
\bigoplus_q R^0f_{\ast}\cH^{n,q}(E)[-q]
\longrightarrow
f_{\ast}\cA_X^{n,\bullet}(E)
\end{equation}
by taking direct sum for all $q$.
Combining with the isomorphism
\begin{equation}
\bigoplus_q R^qf_{\ast}(\omega_X \otimes E)[-q]
\longleftarrow
\bigoplus_q R^0f_{\ast}\cH^{n,q}(E)[-q] ,
\end{equation}
we obtain an isomorphism
\begin{equation}
\bigoplus_q R^qf_{\ast}(\omega_X \otimes E)[-q]
\longleftarrow
\bigoplus_q R^0f_{\ast}\cH^{n,q}(E)[-q] ,
\longrightarrow
f_{\ast}\cA_X^{n,\bullet}(E)
\simeq
Rf_{\ast}(\omega_X \otimes E)
\end{equation}
in the derived category as desired.
\end{proof}

As a corollary of the theorem above,
we have the following:

\begin{cor}
\label{corollary}
In addition to the situation in Theorem \ref{main theorem},
let $g: Y \longrightarrow Z$ be any morphism of \ca spaces.
Then we have
\begin{equation}
\bigoplus_{p+q=n}
R^pg_{\ast}R^qf_{\ast}(\omega_X \otimes E)
\simeq
R^n(g \cdot f)_{\ast}(\omega_X \otimes E)
\end{equation}
for every $n$.
In particular, we have
\begin{equation}
\label{decomposition of cohomology groups}
\bigoplus_{p+q=n}
\coh^p(Y, R^qf_{\ast}(\omega_X \otimes E))
\simeq
\coh^n(X, \omega_X \otimes E)
\end{equation}
for every $n$.
\end{cor}

\begin{rmk}
For the case of $X$ being compact,
the decomposition \eqref{decomposition of cohomology groups}
of the cohomology groups
is proved by S. Matsumura \cite[Corollary 1.2]{MatsumuraVT}.
\end{rmk}

\providecommand{\bysame}{\leavevmode\hbox to3em{\hrulefill}\thinspace}
\providecommand{\MR}{\relax\ifhmode\unskip\space\fi MR }
\providecommand{\MRhref}[2]{%
  \href{http://www.ams.org/mathscinet-getitem?mr=#1}{#2}
}
\providecommand{\href}[2]{#2}

\end{document}